\documentclass[11pt]{article}

\usepackage{amsmath,amsfonts,amsthm,amssymb,graphicx,tikz,url,hyperref}
\usepackage[all,2cell,ps]{xy}

\theoremstyle{plain}
\newtheorem{thm}{Theorem}[section]

\newtheorem{lem}[thm]{Lemma}
\newtheorem{prop}[thm]{Proposition}

\newtheorem{cor}[thm]{Corollary}

\theoremstyle{definition}
\newtheorem{defn}[thm]{Definition}

\newtheorem{rem}[thm]{Remark}

\theoremstyle{remark}

\newcommand{\bbA}{\mathbb{A}}
\newcommand{\bbB}{\mathbb{B}}
\newcommand{\bbC}{\mathbb{C}}

\newcommand{\bbP}{\mathbb{P}}
\newcommand{\bbQ}{\mathbb{Q}}
\newcommand{\bbR}{\mathbb{R}}

\newcommand{\bbZ}{\mathbb{Z}}

\newcommand{\bfO}{\mathbf{O}}

\newcommand{\calD}{\mathcal{D}}

\newcommand{\calG}{\mathcal{G}}

\newcommand{\calL}{\mathcal{L}}

\newcommand{\calN}{\mathcal{N}}
\newcommand{\calO}{\mathcal{O}}

\newcommand{\al}{\alpha}
\newcommand{\gam}{\gamma}
\newcommand{\Gam}{\Gamma}
\newcommand{\del}{\delta}

\newcommand{\lam}{\lambda}
\newcommand{\Lam}{\Lambda}
\newcommand{\sig}{\sigma}

\newcommand{\om}{\omega}

\DeclareMathOperator{\U}{U}

\DeclareMathOperator{\M}{M}

\DeclareMathOperator{\PSL}{PSL}
\DeclareMathOperator{\GL}{GL}

\DeclareMathOperator{\PU}{PU}

\DeclareMathOperator{\Sp}{Sp}

\DeclareMathOperator{\tr}{tr}

\DeclareMathOperator{\Aut}{Aut}

\DeclareMathOperator{\Hom}{Hom}

\newcommand{\bs}{\backslash}

\newcommand{\lra}{\longrightarrow}

\newcommand{\nil}{\varnothing}
\newcommand{\ssm}{\smallsetminus}

\newcommand{\conj}{\overline}

\newcommand{\wt}{\widetilde}

\newenvironment{pf}{\begin{proof}}{\end{proof}}

\newenvironment{enum}{\begin{enumerate}}{\end{enumerate}}

\usepackage{tikz-cd}
\allowdisplaybreaks

\makeatletter
\let\@@pmod\pmod
\DeclareRobustCommand{\pmod}{\@ifstar\@pmods\@@pmod}
\def\@pmods#1{\mkern4mu({\operator@font mod}\mkern 6mu#1)}
\makeatother

\usetikzlibrary{arrows}

\begin{document}

\title{Residual finiteness for central extensions of lattices in $\PU(n,1)$ and negatively curved projective varieties}

\author{Matthew Stover\\ \small{Temple University}\\ \small{\textsf{mstover@temple.edu}} \and Domingo Toledo\\ \small{University of Utah}\\ \small{\textsf{toledo@math.utah.edu}}}

\date{\today}

\maketitle

\begin{center}
\vspace{-1.5em}\emph{To Herb Clemens in celebration of his many contributions to our understanding and enjoyment of the topology of algebraic varieties}
\end{center}

\begin{abstract}
We study residual finiteness for cyclic central extensions of cocompact arithmetic lattices $\Gam < \PU(n,1)$ of simple type. We prove that the preimage of $\Gam$ in any connected cover of $\PU(n,1)$, in particular the universal cover, is residually finite. This follows from a more general theorem on residual finiteness of extensions whose characteristic class is contained in the span in $H^2(\Gam, \bbZ)$ of the Poincar\'e duals to totally geodesic divisors on the ball quotient $\Gam \bs \bbB^n$. For $n \ge 4$, if $\Gam$ is a congruence lattice, we prove residual finiteness of the central extension associated with any element of $H^2(\Gam, \bbZ)$.

Our main application is to existence of cyclic covers of ball quotients branched over totally geodesic divisors. This gives examples of smooth projective varieties admitting a metric of negative sectional curvature that are not homotopy equivalent to a locally symmetric manifold. The existence of such examples is new for all dimensions $n \ge 4$.
\end{abstract}

\section{Introduction}\label{sec:Intro}

It is a classical result that if $\Gam < \PSL_2(\bbR) \cong \PU(1,1)$ is a lattice and $\wt{\Gam}$ denotes its preimage in the universal covering group $\wt{\PSL}_2(\bbR) \cong \wt{ \PU}(1,1)$, then $\wt{\Gam}$ is residually finite. In fact, $\wt{\Gam}$ is a linear group, i.e., it admits a faithful representation into $\GL_N(\bbR)$ for some $N$, even though its ambient Lie group $\wt{\PSL}_2(\bbR)$ is not a linear Lie group. See \cite[\S IV.48]{delaHarpe} for an account of several closely-related perspectives on how one can prove these results. In the other direction, following work of Deligne \cite{Deligne}, one can use the congruence subgroup property to build many lattices in nonlinear Lie groups that are not linear or residually finite, e.g., the preimage of $\Sp_{2n}(\bbZ)$ in the universal cover of $\Sp_{2n}(\bbR)$ for $n \ge 2$.
 
These examples illustrate the interest in studying linearity of lattices in nonlinear Lie groups. Particularly interesting are the universal covering groups $\wt{\PU}(n,1)$ of $\PU(n,1)$, $n \ge 2$. In recent work \cite{StoverToledo}, we developed tools for proving residual finiteness of preimages of lattices in $\PU(n,1)$ in $\wt{\PU}(n,1)$, each inspired by a variant of the proof for $\PU(1,1)$, and found the first examples of residually finite lattices in $\wt{\PU}(2,1)$. One purpose of this paper is to vastly improve upon the class of examples with this property. Recall that arithmetic lattices in $\PU(n,1)$ can be classified using hermitian forms on $\calD^r$, where $\calD$ is a certain central simple division algebra with involution of second kind. In this paper we consider those of \emph{simple type}, namely those constructed using hermitian forms over number fields; see \S\ref{ssec:Simple} for a precise description. We will prove:

\begin{thm}\label{thm:MainRF}
Let $\Gam < \PU(n,1)$ be a cocompact arithmetic lattice of simple type. Then the preimage of $\Gam$ in $\wt{\PU}(n,1)$ is residually finite. In fact, the preimage of $\Gam$ in any connected cover of $\PU(n,1)$ is residually finite.
\end{thm}

This answers \cite[Qu.~1]{StoverToledo} for these lattices. Moreover, the method of proof completely answers \cite[Qu.~2]{StoverToledo}. While completing this manuscript, we learned that Richard Hill also very recently proved Theorem~\ref{thm:MainRF} using related but ultimately quite different methods \cite{Hill}. The proof of Theorem~\ref{thm:MainRF} is in \S\ref{sec:Proofs}, where we prove the following much more general result about residual finiteness of central extensions of lattices in $\PU(n,1)$.

\begin{thm}\label{thm:MainRF2}
Suppose that $\Gam \bs \bbB^n$ is a smooth compact ball quotient with $\Gam$ a congruence arithmetic lattice of simple type. Let $\wt{\Gam}$ be the central extension of $\Gam$ by $\bbZ$ with characteristic class $\phi \in H^2(\Gam, \bbZ)$.
\begin{enum}

\item If $n\ge 4$, then $\wt{\Gam}$ is residually finite.

\item If $n <4$, then $\wt{\Gam}$ is residually finite under the following additional assumptions:

\begin{enum}

\item If $n = 3$, assume that $\phi \in H^{1,1}(\Gam \bs \bbB^3, \bbC) \cap H^2(\Gam \bs \bbB^3, \bbZ)$.

\item If $n = 2$, assume that $\phi$ is contained in the span of the Poincar\'e duals to the totally geodesic divisors on $\Gam \bs \bbB^2$.

\end{enum}

\end{enum}
Moreover, if $\conj{\phi} \in H^2(\Gam, \bbZ / d)$ is the reduction modulo $d$ of any class $\phi$ in $H^2(\Gam, \bbZ)$ satisfying the above hypotheses, then the central extension of $\Gam$ by $\bbZ / d$ associated with $\conj{\phi}$ is residually finite.
\end{thm}

We now describe why Theorem~\ref{thm:MainRF} is a consequence of Theorem~\ref{thm:MainRF2}. Residual finiteness is a commensurability invariant, so it suffices to assume that $\Gam$ is a torsion-free congruence arithmetic lattice. Then, the canonical class $c_1(\Gam \bs \bbB^n) \in H^2(\Gam \bs \bbB^n, \bbZ)$ is the characteristic class associated with the preimage $\wt{\Gam}$ of $\Gam$ in $\wt{\PU}(n,1)$, and similarly the reduction $\wt{\Gam}_d$ of $\wt{\Gam}$ modulo $d$ is isomorphic to the preimage of $\Gam$ in the $d$-fold connected cover of $\PU(n,1)$. Under our hypotheses, $c_1(\Gam \bs \bbB^n)$ is contained in the span of the Poincar\'e duals to the totally geodesic divisors on $\Gam \bs \bbB^n$; see Theorem~\ref{thm:ChernInSpan}. The proof uses Kudla--Millson theory \cite{KM}, following an analogous argument of Bergeron, Li, Millson, and Moeglin for orthogonal Shimura varieties \cite[Cor.~8.4]{Invent}. Thus $c_1(\Gam \bs \bbB^n)$ satisfies the hypotheses of Theorem~\ref{thm:MainRF2}, and Theorem~\ref{thm:MainRF} follows.

The proof of Theorem~\ref{thm:MainRF2} begins by using work of Bergeron, Millson, and Moeglin \cite{Acta} to show that for any class satisfying the hypotheses of the theorem, there is a congruence subgroup $\Gam^\prime \le \Gam$ so that the pullback of $\phi$ to $\Gam^\prime \bs \bbB^n$ is in the image of the cup product from $\bigwedge\nolimits^2 H^1(\Gam^\prime \bs \bbB^n, \bbC)$; this is Theorem~\ref{thm:Cup}. Our previous work \cite[Thm.~5.1]{StoverToledo}, stated in slightly different language as Theorem~\ref{thm:CentralRF1} below, then implies that the associated central extension $\wt{\Gam}$ of $\Gam$ by $\bbZ$ has a two-step nilpotent quotient that is injective on the center of $\wt{\Gam}$. The claims regarding residual finiteness of $\wt{\Gam}$ and $\wt{\Gam}_d$ follow from results in \cite[\S 2]{StoverToledo} that we recall in \S\ref{sub:residualf} below.

\begin{rem}\label{rem:Linear}
As in \cite{StoverToledo}, we also obtain that every extension considered in Theorem~\ref{thm:MainRF2} is linear. Our proofs of residual finiteness proceed by constructing a nilpotent quotient of the extension group that is injective on the center. Using linearity of finitely generated nilpotent groups we easily get linearity of the extension. We state only the weaker results because residual finiteness is the property of interest to us for our applications.
\end{rem}

Taking the special case in Theorem~\ref{thm:MainRF2} where $\phi$ is the Poincar\'e dual to a codimension one totally geodesic subvariety, we have the following.

\begin{cor}\label{cor:ComplementRF}
Let $M = \Gam \bs \bbB^n$ be a closed complex hyperbolic $n$-manifold with $\Gam$ a congruence arithmetic lattice. If $D \subset M$ is a smooth embedded codimension one totally geodesic subvariety, $\calO(D)$ is the line bundle over $M$ associated with $D$, and $\calO(D)^\times \subset \calO(D)$ is the complement of the zero section, then $\pi_1(\calO(D)^\times)$ is residually finite.
\end{cor}

If $\Gam < \PU(n,1)$ is an arithmetic lattice so that $\Gam \bs \bbB^n$ contains an embedded codimension one totally geodesic submanifold, then $\Gam$ is of simple type. See Proposition~\ref{prop:TGSimple}, which generalizes an argument from work of M\"oller and Toledo \cite{MollerToledo}. All known nonarithmetic ball quotients of dimension $n \ge 2$ are well known to be commensurable with quotients of the ball by \emph{complex reflection groups}, which implies that they contain codimension one totally geodesic subvarieties (e.g., see \cite[Thm.~1.3]{StoverTriangle}, where the proof works verbatim in all dimensions). Thus residual finiteness of $\pi_1(\calO(D)^\times)$ is relevant, and remains open, for nonarithmetic lattices. We note for contrast that arithmetic ball quotients of simple type contain infinitely many distinct immersed codimension one totally geodesic subvarieties, but recent work of Bader--Fisher--Miller--Stover \cite{BFMS2} and, independently, Baldi--Ullmo \cite{BaldiUllmo} proves that nonarithmetic ball quotients only contain finitely many immersed codimension one totally geodesic subvarieties (and, more generally, only finitely many of any positive dimension that are maximal with respect to inclusion).

\medskip

Our final application of Theorem~\ref{thm:MainRF2} is to the existence of covers of ball quotients branched over totally geodesic divisors, answering a question raised to the second author by Gromov over $40$ years ago. Existence of these covers allows us to prove the following in \S\ref{sec:Branched}.

\begin{thm}\label{thm:NegativeExist}
For all $n \ge 2$, there are $n$-dimensional smooth complex projective varieties admitting a K\"ahler metric of strongly negative curvature, thus a Riemannian metric of strictly negative sectional curvature, that are not homotopy equivalent to a locally symmetric manifold.
\end{thm}

Strongly negative curvature is in the sense defined by Siu \cite[\S 2]{Siu}. Examples for $n=2$ are due to Mostow--Siu \cite{MostowSiu}, and further examples were produced by Zheng \cite{Zheng, Zheng2}. Deraux gave examples for $n = 3$ \cite{Deraux}. We contribute a wealth of new examples for $n=2,3$ and the first examples for $n \ge 4$. Our examples are cyclic covers of ball quotients branched over smooth (but possibly disconnected) totally geodesic divisors. Existence of these covers follows from Corollary~\ref{cor:ComplementRF}. The fact that they admit a strongly negatively curved K\"ahler metric is a theorem of Zheng \cite[Thm.~1]{Zheng} that generalizes the famous work of Mostow and Siu \cite{MostowSiu}.   Zheng notes in \cite[p.~135~\&~151]{Zheng} that existence of branched covers was a primary obstruction to proving Theorem~\ref{thm:NegativeExist} in higher dimensions.

Finally we  must prove that these covers are not homotopy equivalent to a locally symmetric space. 
This was shown for the previous examples by computing Chern numbers.   This becomes delicate in higher dimensions, so we provide a proof of a different kind. Assuming our branched cover is homotopy equivalent  to a locally symmetric manifold, hyperbolicity of the fundamental group, work of Carlson and Toledo \cite{CarlsonToledo}, and Siu rigidity \cite{Siu} allows us to conclude that the manifold is in fact biholomorphic to a ball quotient. We then study the normal bundle to the branch locus to derive a contradiction.

\subsubsection*{Acknowledgments}
Profound thanks are due to Nicolas Bergeron, who explained to us how to apply Kudla--Millson theory and his work to prove the results on cup products necessary to apply our previous methods. In particular, the arguments for Theorem~\ref{thm:ChernInSpan} and Theorem~\ref{thm:Cup} should be attributed to him. We also thank the referee for a number of helpful suggestions that improved the paper. Stover was partially supported by Grant Number DMS-1906088 from the National Science Foundation.

\section{Residual finiteness, central extensions, and\\ branched covers}\label{sec:RFetc}

\subsection{Basic objects}\label{subsec:basic}

We review the connections among the following objects. Suppose:
\begin{enum}
\item $X$ is a smooth aspherical manifold with $\Gamma = \pi_1(X)$;

\item $p: L \to X$ is a complex line bundle over $X$, $L^\times$ is $L$ with its zero section removed, and $\wt{\Gam} = \pi_1(L^\times)$;

\item $s:X\to L$ is a section of $L$ transverse to the zero section. 

\end{enum}
The homotopy sequence of the fibration $p : L^\times \to X$ with fiber $\bbC^\times$ gives a central extension
\[
0 \lra \bbZ \lra \wt{\Gam} \lra \Gam \lra 1
\]
connecting the above fundamental groups. The zero set
\[
Z(s) = \{x\in X: s(x) = 0\}
\]
of $s$ is a smooth real codimension two submanifold of $X$ with normal bundle isomorphic to $L|_{Z(s)}$, where $L$ is viewed as an $\bbR^2$-bundle.

\subsubsection{Group of central extensions }
\label{subsub:extensions}

It is well known that for a given group $\Gam$ and fixed Abelian group $A$ the isomorphism classes of central extensions
\[
E: \, 0 \lra A \lra \wt{\Gam} \lra \Gam \lra 1
\]
form a group in natural bijection with $H^2(\Gamma, A)$, where the correspondence is given by the characteristic class $\chi(E)$ of the extension $E$; e.g., see \cite[\S VI.3]{Brown}.  Briefly, $\chi(E)$ is represented by the following Eilenberg--MacLane cocycle $c$ whose cohomology class uniquely determines $E$:  choose a set theoretic section $s:\Gam \to \wt{\Gam}$, let $c(\gam_1,\gam_2) = s(\gam_2)s(\gam_1 \gam_2)^{-1}s(\gam_1).$ This is equivalent to writing $\wt{\Gam}$ as the product $\Gam \times A$ with the multiplication
\[
(\gam_1, a_1)(\gam_2,a_2) = (\gam_1 \gam_2 , a_1 + a_2 + c(\gam_1,\gam_2)).
\]
This correspondence is functorial in both of the variables $\Gam$ and $A$ and the resulting operations on extensions are best defined in terms of $\chi(E)$.  Two common examples: given an injection ${\iota : \Lam\to \Gam}$, then the restriction $\iota^*(E)$ of $E$ to $\Lam$ has characteristic class $\iota^*(\chi(E))$. Similarly, if $p:A\to B$ is surjective and $p_*(E)$ denotes the extension of $\Gam$ obtained by dividing $A$ and $\wt{\Gam}$ by the kernel of $p$, then $\chi(p_*(E)) = p_*(\chi(E))$.

\begin{rem}
\label{rem:ext}
Note that the word \lq\lq extension" is used in two (closely related) meanings: the exact sequence $E$ above, or just the middle group $\wt{\Gam}$. We will refer to $\wt{\Gam}$ as \emph{the group of the extension $E$}. Given $\iota : \Lam \to \Gam$ as above, the group $\wt{\Lam}$ of the extension $\iota^*(E)$ is then the preimage of $\Lam$ in $\wt{\Gam}$ for the natural projection $\wt{\Gam} \to \Gam$.
\end{rem}

\subsubsection{Group of line bundles }
\label{subsub:linebdles}

It is also well known that given a space $X$, the isomorphism classes of complex line bundles over $X$ form a group under tensor product that is isomorphic to $H^2(X,\bbZ)$, with isomorphism given by the Chern class $c_1(L)$. The isomorphism is natural with respect to smooth maps $f:Y\to X$, namely $c_1(f^*L) = f^*(c_1(L))$.

\subsubsection{Comparison of characteristic classes}
\label{subsub:comparison}

Suppose that we are in the above situation: $X$ is a smooth aspherical manifold and $\Gam =\pi_1(X)$, hence $H^*(X) \cong H^*(\Gam)$ (meaning that there is an isomorphism of cohomology with any coefficients). In order to compare the two characteristic classes defined above, we need a functorial isomorphism between the two cohomologies. More precisely, we need a natural isomorphism between cohomology of the complex $C^*(\Gam,\bbZ)$ of Eilenberg--MacLane cochains on $\Gam$ and the cohomology of a topologically defined complex $C^*(X,\bbZ)$ of cochains on $X$, where one could choose either singular or \v{C}ech cochains on $X$. There is a well known method for establishing a natural correspondence between the cohomologies of these two complexes by embedding both in the bicomplex $\bigoplus_{p,q}C^p (\Gam,C^q(\wt{X},\bbZ))$ of Eilenberg--MacLane cochains on $\Gam$ with coefficients in the (singular or \v{C}ech) cochains on the universal cover $\wt{X}$ with its action of $\Gam$. Standard arguments give natural isomorphisms from the two cohomologies in question to that of the total complex of the bicomplex. The isomorphism thus obtained will give a functorial isomorphism between $H^2(\Gam,\bbZ)$ and $H^2(X,\bbZ)$.

Thus, in our situation of a complex line bundle $L$ over a smooth aspherical manifold $X$ we have three cohomology classes in $H^2(X,\bbZ)\cong H^2(\Gam,\bbZ)$:
\begin{enum}

\item The characteristic class $\chi(E)\in H^2(\Gam,\bbZ)$ of the extension
\begin{equation}\label{eq:EZ}
E: \, 0 \lra \bbZ \lra \wt{\Gam} \lra \Gam \lra 1,
\end{equation}
where $\wt{\Gam} = \pi_1(L^\times)$.

\item The Chern class $c_1(L)\in H^2(X,\bbZ)$.

\item The Poincar\'e dual $[Z(s)]^\vee$ of the homology class $[Z(s)]\in H_{n-2}(X,\bbZ)$ of the zero set $Z(s)$ of the section $s$.

\end{enum}

\begin{lem}
\label{lem:compare}
The three cohomology classes defined in 1\,-\,3 above are equal, (except possibly up to a factor of $\pm 1$ arising from sign conventions).
\end{lem}

\begin{proof}
The equivalence of (2) and (3) is standard; see, for example, \cite[Prop.~6.24]{BottTu} or \cite[Ch.~1, Prop.~1]{GriffithsHarris}. The proof of the equivalence of (1) and (2) should also be standard. Since we could not find a reference we give a quick sketch.

We start from the fact that the group-cochain analogue of the standard \v{C}ech cochain representative of $c_1(L)$ gives its representative in group cohomology. Namely, let $p:\wt{X}\to X$ be the universal cover, with $\Gam$ acting on $\wt{X}$ by covering transformations. Since $p^*L^\times$ is trivial, it has a nowhere vanishing section $s$. For each $\gam\in\Gam$ define a function $a_\gam:\wt{X}\to \bbC^\times$ by $\gam^* s = a_\gam s$. There exist functions $\ell_\gam:\wt{X}\to \bbC$, unique up to integral additive constant, such that $\exp(2\pi i \ell_\gam) = a_\gam$, and these define $\ell \in C^1(\Gam, C^1(\wt{X}, \bbZ))$ as above. Then the Eilenberg--MacLane coboundary $\del\ell$ given by
\[
\del\ell(\gam_1,\gam_2) = \ell(\gam_2) - \ell(\gam_1\gam_2) + \gam_2^*\ell(\gam_1)
\]
is a $\bbZ$-valued cocycle representing $c_1(L)\in H^2(\Gam,\bbZ)\cong H^2(X,\bbZ)$.
 
To show the equivalence of (1) and (2) it suffices to show that $\del \ell$ is a cocycle for the extension in Equation \eqref{eq:EZ}.  To do this we represent $\wt{\Gam}$ as the Deck transformation group of the universal cover $\wt{L}^\times\cong \wt{X}\times \bbC$. Writing $\wt{\Gam}$ set-theoretically as $\Gam\times \bbZ$, we see that $(\gam,n)\in \Gam\times \bbZ$ acts on $(x,t)\in \wt{X}\times \bbC$ by
\[
(\gam,n)(x,t) =( \gam x , t + n + \ell_\gam (x)).
\]
From this it is straightforward to deduce that the cocycle giving the set $\Gam\times \bbZ$ the group structure of $\wt{\Gam}$ is $\del\ell$.
\end{proof}

\subsubsection{The same considerations over projective varieties}
\label{subsub:projective}

Now suppose, in addition to the above assumptions, that $X$ is a smooth projective variety. Let $D = D_1 \cup \dots \cup D_k$, where $D_1,\dots ,D_k$ are smooth and connected codimension one subvarieties of $X$ that are pairwise disjoint, i.e., $D_i \cap D_j = \nil$ if $i\neq j$. We will not distinguish the smooth codimension one subvariety $D$ from the smooth divisor $ D = D_1 + \dots + D_k$.

\begin{defn}
\label{def:goodpair}
A pair $(X,D)$ consisting of a smooth projective variety $X$ and a divisor $D$ satisfying the above conditions will be called a \emph{good pair}.
\end{defn}

We can assign a holomorphic line bundle $\calO(D) \to X$ to $D$ with the property that it has a holomorphic section $s$ with zero set $Z(s) = D$ and vanishing to first order on $D$, meaning that, in local coordinates on $X$ and $\calO(D)$, the differential $d_x s$ is nonzero for any $x\in D$. These properties define $\calO(D)$ uniquely, and its section $s$ uniquely up to a nonzero multiplicative constant, i.e., we can replace $s$ by $\lam s$ for any $\lam\in\bbC^\times$. Transversality of $s$ implies that $\calO(D)|_D$ is isomorphic to the normal bundle $N(D)$ of $D$ in $X$.

Let $d>1$ be an integer and suppose that $\calO(D)$ is divisible by $d$ in the group of holomorphic line bundles on $X$. That is, suppose that there is a holomorphic line bundle $p:L \to X$ with $L^{\otimes^d}\cong \calO(D)$. We have a natural $d^{th}$ power map $L\to L^{\otimes^d} $ defined by
\[
v \longmapsto v^d = \overset{d~\textrm{factors}}{\overbrace{v \otimes \dots \otimes v}}.
\]

\begin{defn}
\label{def:branchedcov}
Suppose $(X,D)$ is a good pair and $p:L \to X$ a line bundle with $L^{\otimes^d} \cong \calO(D)$ as above. Then the space $Y$ defined by
\[
Y = \left\{ v\in L: v^d = s(p(v))\right\}
\]
will be called a \emph{cyclic $d$-fold branched cover of $X$ branched along $D$}.
\end{defn}

More precisely, $Y$ could be called the cyclic branched cover associated with $L$. Note the following properties of $Y$:

\begin{lem}
\label{lem:propsbranch}
Given a good pair $(X, D)$ and $L$ a line bundle so $L^{\otimes^d} \cong \calO(D)$, the cyclic branched cover $Y$ branched along $D$ satisfies:
\begin{enum}
\item As a subset of $L$, $Y$ is a smooth projective subvariety stable under the multiplication maps $v\mapsto \zeta v$ where $\zeta$ is a primitive $d^{th}$ root of unity, which define an action of the cyclic group $\bbZ/d$ on $Y$ with quotient $X$.

\item The projection $p:Y\to X$ is $d$-to-$1$ over $X \ssm D$ and $1$-to-$1$ over $D$.

\item If $\bfO$ denotes the zero section of $L$, then
\[
p|_{Y\ssm\bfO} : Y\ssm \bfO \to X \ssm D
\]
is a $\bbZ/d$ covering space.

\end{enum}

\end{lem}

\begin{proof}
By writing local equations for $Y$ at each point $y\in Y\subset L$ it is clear that $Y$ is a smooth subvariety. Since $L$ is a quasiprojective variety, so is $Y$, and $Y$ is finite over $X$, hence compact, so $Y$ is projective. The defining equation of $Y$ is invariant under the $\bbZ/d$-action, which is free on $L\ssm \bfO$ and fixes $\bfO$ point-wise. This makes the remaining statements clear.
\end{proof} 

For the existence of branched covers as in Definition \ref{def:branchedcov} we have:

\begin{prop}
\label{prop:existroots}
Suppose $(X,D)$ is a good pair and $d > 1$ is an integer. With notation as in \S\ref{subsub:extensions}, the following are equivalent:
\begin{enum}

\item There exists a holomorphic line bundle $p:L\to X$ with $L^{\otimes^d}\cong \calO(D)$.

\item There exists $\alpha\in H^2(X,\bbZ)$ with $c_1(\calO(D)) = d\cdot \al$.

\item There exits $\gam\in H^2(\Gam,\bbZ)$ with $\chi(E) = d\cdot \gam$, where $\Gam =\pi_1(X)$, ${\wt{\Gam} = \pi_1(\calO(D)^\times)}$, and $E$ is the extension in Equation~\eqref{eq:EZ}.

\item There exists $z\in H_{2n-2}(X,\bbZ)$ with $d\cdot z = [D]$, where $n = \dim_\bbC(X)$ and $[D]$ is the homology class of $D$.

\item The class $(r_d)_* (c_1(\calO(D))) \in H^2(X,\bbZ/d)$ vanishes, where
\[
(r_d)_*:H^2(X,\bbZ)\lra H^2(X,\bbZ/d)
\]
is the map induced by the coefficient homomorphism ${r_d:\bbZ\to\bbZ/d}$.

\item The class $\chi((r_d)_*(E))\in H^2(\Gamma,\bbZ/d)$ vanishes, where $(r_d)_*(E)$ is the extension
\[
0\lra \bbZ/d \lra \wt{\Gam}_d \lra \Gam\lra 1
\]
for $\wt{\Gam}_d = \wt{\Gam} / d\bbZ$ obtained from $E$ by dividing by the kernel of $r_d$.

\end{enum}
If any one of these conditions holds, then there exists a cyclic $d$-fold branched cover $Y\to X$ branched along $D$ (associated with $L$) as in Definition \ref{def:branchedcov}.
\end{prop}

\begin{proof}
Clearly (1) implies (2) with $\alpha = c_1(L)$, and (2) and (4) are equivalent by Poincar\'e duality and Lemma~\ref{lem:compare}. Also (2) and (5) are seen to be equivalent by looking at the Bockstein long exact sequence associated with the coefficient sequence $0\to \bbZ \to \bbZ\to \bbZ/d\to 0$. The equivalence of (2) and (3), and similarly (5) and (6), follow from Lemma \ref{lem:compare}. The only subtle point is to prove that (2) implies (1). We refer to \cite[\S 7]{HirzebruchRam} for a simple proof based on the exponential sequence $0\to\bbZ\to \calO \to \calO^*\to 1$.
\end{proof}

\subsection{Residual Finiteness}
\label{sub:residualf}

Recall that a finitely generated group $\Gam$ is \emph{residually finite} if, for every nontrivial $\gam \in \Gam$, there exists a homomorphism $\rho : \Gam \to F$ from $\Gam$ to a finite group $F$ so that $\rho(\gam)$ is nontrivial. This property is a commensurability invariant of groups. In particular, if $\Lam\le \Gam$ is a subgroup of finite index, then $\Lam$ is residually finite if and only if $\Gam$ is residually finite.

We will be interested in residual finiteness for the groups $\wt{\Gam}$ of central extensions
$0\to \bbZ\to \wt{\Gam} \to\Gam\to 1$ of residually finite groups $\Gam$ by $\bbZ$. Throughout this section $q$ will denote the projection $\wt{\Gam}\to\Gam$ and $\sig$ will be the image of $1$ under the injection $\bbZ\to\wt{\Gam}$, hence the infinite cyclic group $\langle \sig \rangle$ is the kernel of $q$.

Since $\Gam$ is residually finite, to prove that $\wt{\Gam}$ is residually finite it suffices to prove that for every $i\ne 0$ there exists a homomorphism from $\wt{\Gam}$ to some finite group that is nontrivial on $\sig^i$. Indeed, for any $\gam\in\wt{\Gam}$ not in $\langle \sig \rangle$, $q(\gam)$ is nontrivial and, by assumption, there exists a homomorphism $\rho:\Gam\to F$ for some finite group $F$ with $\rho(q(\gam))$ nontrivial. Thus $\rho\circ q:\wt{\Gam}\to F$ is a homomorphism that is nontrivial on $\gam$. Observe that a slight extension of this reasoning gives:

\begin{lem}[Lem.~2.4 \cite{StoverToledo}]
\label{lem:rfquotient}
Let $E$ be a central extension of a residually finite group $\Gam$ by $\bbZ$ with group $\wt{\Gam}$ as above, and let $\sig\in\wt{\Gam}$ be the image of $1\in\bbZ$. If there is a homomorphism $\phi:\wt{\Gam}\to H$ where $H$ is residually finite and $\phi$ is injective on $\langle\sig\rangle$, then $\wt{\Gam}$ is residually finite.
\end{lem}

For $d\in\bbZ$, $d>1$, we will also be interested in the reductions modulo $d$ of the above extensions, namely:
\begin{equation}\label{eq:Ed}
0\lra \bbZ/d\lra\wt{\Gam}_d\lra\Gam\lra 1,
\end{equation}
where $\wt{\Gam}_d = \wt{\Gam}/d\bbZ$. Similar considerations apply with regard to residual finiteness of $\wt{\Gam}_d$. The following definition will turn out to be very useful:

\begin{defn}
\label{def:conditionn}
Let $\Gam$ be a residually finite group, $E$ be an extension of $\Gam$ by $\bbZ$ with group $\wt{\Gam}$ as in Equation~\eqref{eq:EZ}, and let $\sig\in\wt{\Gam}$ be the image of $1\in\bbZ$. We say that \emph{$E$ satisfies condition $N$} if there exists a nilpotent quotient $\calN$ of $\wt{\Gam}$ under which the image of $\sig$ has infinite order. For a positive integer $k$ we say that \emph{$E$ satisfies condition $N_k$} if there is a nilpotent quotient $\calN$ of $\wt{\Gam}$ with step size at most $k$ in which the image of $\sig$ has infinite order.
\end{defn}

\begin{lem}
\label{lem:condn}
Let $E$ be a central extension of a residually finite group $\Gam$ by $\bbZ$ with group $\wt{\Gam}$. If $E$ satisfies condition $N$, then $\wt{\Gam}$ is residually finite.
\end{lem}

\begin{proof}
Since finitely generated nilpotent groups are residually finite, this is an immediate consequence of Lemma \ref{lem:rfquotient}.
\end{proof}

First we need to know how residual finiteness behaves under the two basic operations of restriction (or pullback) and reduction (push-forward) of extensions defined in \S\ref{subsub:extensions}.

\begin{lem}
\label{lem:behavior}
Let $\Gam$ be a residually finite group and $E$ denote the central extension $ 0 \to \bbZ \to \wt{\Gam}\to\Gam\to 1$. If $d$ is an integer $>1$, let $r_d:\bbZ\to \bbZ/d$ denote reduction mod $d$ and let $E_d$ denote the central extension $(r_d)_*(E)$.
\begin{enum}

\item Let $\Lam\le\Gam$ be a subgroup of finite index and $\iota:\Lam \to\Gam$ be the inclusion map. If $\wt{\Lam}$ denotes the group of the central extension $\iota^*(E)$, then $\wt{\Lam}$ is residually finite if and only if $\wt{\Gam}$ is residually finite.

\item The group $\wt{\Gam}_d = \wt{\Gam}/ (d\bbZ)$ is residually finite for infinitely many $d>1$ if and only if  $\wt{\Gam}$ is residually finite.

\item Suppose $E$ satisfies condition N of Definition \ref{def:conditionn}. Then the group $\wt{\Gam}_d$ of $(r_d)_*(E)$ is residually finite for all $d \ge 1$. 
\end{enum} 

\end{lem}

\begin{proof}{\ }
\begin{enum}
\item Since $\wt{\Lam}$ is a finite index subgroup of $\wt{\Gam}$, this equivalence is a standard property of residual finiteness.

\item As explained before the statement of Lemma~\ref{lem:rfquotient}, to prove residual finiteness of $\wt{\Gam}$ we only need to separate nontrivial elements of the kernel of $q$, i.e., the infinite cyclic group $\langle \sig \rangle$, from the identity. Consider an element $\sig^i$, $i \in\bbZ \ssm \{0\}$, and choose any integer $d>|i|$ so that $\wt{\Gam}_d$ is residually finite. Then $(r_d)_*(\sig^i)$ is nontrivial in $\wt{\Gam}_d$ and there is a homomorphism $\phi:\wt{\Gam}_d\to F$ for some finite group $F$ with $\phi(\sig^i )$ nontrivial. Then $\phi\circ (r_d):\wt{\Gam}\to F$ has $\phi\circ(r_d)(\sig^i)$ nontrivial, hence $\wt{\Gam}$ is residually finite.

For the converse, given any $i > 0$ there is a finite quotient $\phi : \wt{\Gam} \to F_i$ so $\phi(\sig^i)$ is nontrivial. Let $d(i)$ be the order of $\phi(\sig^i)$. Then $\phi$ factors through $\wt{\Gam}_{d(i)}$, and it follows from the mod $d(i)$ analogue of Lemma~\ref{lem:rfquotient} that $\wt{\Gam}_{d(i)}$ is residually finite. Since $d(i)$ is greater than or equal to the largest prime divisor of $i$, $\displaystyle{\varlimsup_{i \to \infty} d(i) = \infty}$, which proves the converse.

\item The statement is more technical, so is the proof. See \cite[Lem.~2.6]{StoverToledo}.

\end{enum}
\end{proof}

\subsection{Application of condition $N$}
\label{sub:applyn}

We will need the following useful fact about residually finite $\bbZ/d$-extensions.

\begin{prop}
\label{prop:rfdiv}
Suppose the group $\wt{\Gam}_d$ of an extension $E$ of a group $\Gam$ by $\bbZ / d$ as in Equation~\eqref{eq:Ed} is residually finite. Then there is a finite index subgroup $\Lam \le \Gam$ such that $\iota^*(E)$ is the trivial extension, where $\iota : \Lam \hookrightarrow \Gam$ is the inclusion. Equivalently, if $\al\in H^2(\Gam,\bbZ/d)$, then there exists $\iota : \Lam \hookrightarrow \Gam$ of finite index so that $\iota^*(\al)\in H^2(\Lam,\bbZ/d)$ vanishes.
\end{prop}

\begin{proof}
The first conclusion is a restatement of \cite[Lem.~2.2]{StoverToledo}. The second statement is the reformulation in terms of characteristic classes.
\end{proof}

\begin{thm}
\label{thm:ndiv}
Suppose the extension $E$ of $\Gam$ by $\bbZ$ with group $\wt{\Gam}$ satisfies condition $N$ of Definition~\ref{def:conditionn} and that $\Gam$ is residually finite. Let $\eta\in H^2(\Gam,\bbZ)$ denote the characteristic class of $E$. Then for each integer $d>1$ there exists a subgroup $\Lam\leq\Gam$ of finite index, depending on $d$, such that all of the following equivalent conditions hold:
\begin{enum}
\item The extension $\iota^*((r_d)_*(E)): 0\to\bbZ/d\to\wt{\Lam}_d\to\Lam\to 1$ (that is, the restriction to $\Lam$ of the reduction of $E$ modulo $d$), is the trivial extension.

\item The class $\iota^*((r_d)_*(\eta)) = (r_d)_*(\iota^*(\eta)) \in H^2(\Lam,\bbZ/d)$ is trivial.

\item The restriction $\iota^*(\eta)$ is divisible by $d$ in $H^2(\Lam,\bbZ)$, that is, there exists ${\al\in H^2(\Lam,\bbZ)}$ with $\eta = d\cdot \al$.

\end{enum}
\end{thm}

\begin{proof}
Since $E$ satisfies condition $N$, the third statement of Lemma \ref{lem:behavior} gives us that $\wt{\Gam}_d = \wt{\Gam} / d \bbZ$ is residually finite. Proposition \ref{prop:rfdiv} gives us the first statement. The second and third statements are equivalent formulations in terms of characteristic classes.
\end{proof}

\subsection{Existence of branched covers}
\label{sub:existbranched}

Let $(X,D)$ be a good pair in the sense of Definition \ref{def:goodpair} and $\Gam =\pi_1(X)$. We want to construct cyclic $d$-fold branched covers, branched along $D$, in the sense of Definition \ref{def:branchedcov}. Proposition~\ref{prop:existroots} gives a number of equivalent conditions for the existence of such a cover, for example, when $c_1(\calO(D))$ divisible by $d$ in $H^2(\Gam,\bbZ)$.

Given $(X,D)$ and $d$, it may be very difficult to check this divisibility condition. However, there are situations where the following weaker statement is already interesting: within a given class of good pairs $\{(X,D)\}$ prove that there exist many members for which the divisibility condition holds. The class of pairs $(X, D)$ of interest to us is that of arithmetic ball quotients $X$ of simple type and $D$ a totally geodesic divisor on $X$. By \lq\lq many members" we mean that any $(X,D)$ in the class has a finite unramified cover $p:X^\prime \to X$ so that $(X^\prime,p^*(D))$ satisfies the above divisibility condition. Here $p^*(D)$ means the divisor on $X^\prime$ with support $p^{-1}(D)$ such that all connected components have multiplicity one. We record the following simple but necessary fact:

\begin{lem}
\label{lem:pstargood}
Suppose $(X,D)$ is a good pair and $p: X^\prime \to X$ is a finite unramified cover. Then $(X^\prime,p^*(D))$ is a good pair. If $s$ is the canonical section of $\calO(D)$, then $p^*(s)$ is the canonical section of $\calO(p^*(D)) = p^*(\calO(D))$.
\end{lem}

\begin{proof}
Smoothness of $p^{-1}(D)$ and vanishing to first order of $p^*(s)$ on $p^{-1}(D)$ are both clear from the fact that $p: X^\prime \to X$ is locally a biholomorphism.
\end{proof}

\begin{thm}
\label{thm:divide}
Let $(X,D)$ be a good pair, let $d>1$ be an integer, and let:
\begin{enum}
\item $\Gam = \pi_1(X)$ and $\wt{\Gam} =\pi_1(\calO(D)^\times)$;
\item $E: 0\to\bbZ \to \wt{\Gam}\to \Gam \to 1$ be the extension given by the homotopy sequence of the fibration $\calO(D)^\times \to \calO(D)$.

\end{enum}
Suppose that $E$ satisfies condition $N$ of Definition~\ref{def:conditionn}. Then there is a finite unramified cover $p:X^\prime \to X$ so that $c_1(p^*(D))$ is divisible by $d$ in $H^2(X^\prime,\bbZ)$.
\end{thm}

\begin{proof}
This is an immediate consequence of Theorem~\ref{thm:ndiv}, by defining $X^\prime$ to be the covering space of $X$ associated with the subgroup $\Lam\le\Gam$.
\end{proof}

\begin{cor}
\label{cor:existbranch}
Let $(X,D)$ be a good pair and suppose the extension
\[
0\lra \bbZ\lra \pi_1(\calO(D)^\times) \lra \pi_1(\calO(D))\lra 1
\]
satisfies condition $N$. Then there is a finite unramified cover $p: X^\prime \to X$ such that there exists a cyclic $d$-fold cover $Y\to X^\prime$ branched along $p^*(D)$.
\end{cor}

Finally, we give a sufficient condition for certain extensions that occur very frequently to satisfy condition $N$. More precisely, we look at a class of extensions for which we can characterize those that satisfy the stronger condition $N_2$ of Definition~\ref{def:conditionn}. The following theorem was stated in \cite{StoverToledo} under an unnecessary assumption. We remove this assumption and give essentially the same proof, partially incorporating some suggestions from the referee.

\begin{thm}[Thm.~5.1 \cite{StoverToledo}]\label{thm:CentralRF1}
Let $X$ be a closed aspherical manifold with fundamental group $\Gam$ and let $L\to X$ be a complex line bundle with Chern class $\om\in H^2(X,\bbZ)$. Equivalently, let $\wt{\Gam}=\pi_1(L^\times)$ and let $E$ be the extension as in Equation~\eqref{eq:EZ} with group $\wt{\Gam}$ and characteristic class $\om\in H^2(\Gam,\bbZ)$. Let $\om_\bbQ$ denote the image of $\om$ in $H^2(\Gam,\bbQ)$. Then, $E$ satisfies condition $N_2$ of Definition~\ref{def:conditionn} if and only if $\om_\bbQ$ is in the image of the map
\[
c_\bbQ : \bigwedge\nolimits^2 H^1(X, \bbQ) \lra H^2(X, \bbQ)
\]
given by evaluation of the cup product.
\end{thm}

\begin{proof}
Let $pr:\Gam \to \Gam^{ab}$ be the projection of $\Gam$ to its abelianization. Since $pr^*:H^1(\Gam^{ab},\bbQ)\to H^1(\Gam,\bbQ)$ is an isomorphism and $H^2(\Gam^{ab},\bbQ) \cong \bigwedge\nolimits^2 H^1(\Gam^{ab}, \bbQ)$, we see that $\om_\bbQ$ is in the image of $c_\bbQ$ if and only if it is in the image of $pr^*:H^2(\Gam^{ab},\bbQ)\to H^2(\Gam,\bbQ )$, that is, if and only if there exists $\wt{\om}_\bbQ\in H^2(\Gam^{ab}, \bbQ)$ with $pr^*(\wt{\om}_\bbQ) = \om_\bbQ$.  Let $\wt{\Gam^{ab}}$ be the group of the extension of $\Gam^{ab}$ by $\bbQ$ with characteristic class $\wt{\om}_\bbQ$.  Thus the original extension with characteristic class $\om\in H^2(\Gam,\bbZ)$ maps to the extension with characteristic class $\wt{\om}_\bbQ \in H^2(\Gam^{ab},\bbQ)$:
\[
\begin{tikzcd}
0 \arrow{r} & \bbZ \arrow{r} \arrow{d} & \wt{\Gam}\arrow{r} \arrow{d} & \Gam \arrow{r} \arrow{d} & 1\\
0\arrow{r}& \bbQ \arrow{r} & \wt{\Gam^{ab}} \arrow{r} & \Gam^{ab}\arrow{r} & 1 \end{tikzcd}
\]
The group $\wt{\Gam^{ab}}$, as a central extension of an abelian group, is nilpotent of step size at most two, and the central $\bbZ$ in $\wt{\Gam}$ maps injectively to the center of $\wt{\Gam^{ab}}$. Thus the image of $\wt{\Gam}$ in $\wt{\Gam^{ab}}$ is a nilpotent quotient of $\Gam$ satisfying Definition \ref{def:conditionn}, and so the extension $\wt{\Gam}$ satisfies condition $N_2$. Note that $\wt{\Gam}$ satisfies condition $N_1$ if and only if $\om_\bbQ = 0$.

Before proving the converse, it is useful to write down some consequences of the Gysin sequence for the fibration $p: L^\times \to X$ stated in terms of the induced map on fundamental groups  $p:\wt{\Gam}\to \Gam$. We use $\bbQ$ coefficients throughout. The sequence
\[
0 \lra H^1(\Gam) \overset{p^*}{\lra} H^1(\wt{\Gam})\overset{p_*}{\lra} H^0(\Gam)\overset{\om_\bbQ}{\lra} H^2(\Gam)\overset{p^*}{\lra} H^2(\wt{\Gam})
\]
shows that if $\om_\bbQ\ne 0$, then $p^*:H^1(\Gam,\bbQ) \overset{\cong}{\lra} H^1(\wt{\Gam},\bbQ)$ is an isomorphism, equivalently, $p_* : \wt{\Gam}^{ab}/ Tor \overset{\cong}{\lra} \Gam^{ab}/Tor$ is an isomorphism, where $Tor$ denotes the respective torsion subgroup, and the kernel of $p^*:H^2(\Gam,\bbQ)\lra H^2(\wt{\Gam},\bbQ)$ is the subspace $\bbQ\, \om_\bbQ$ of $H^2(\Gam,\bbQ)$.

Using this information to compare the kernels $K_{\Gam}, K_{\wt{\Gam}}$ and the images $I_{\Gam},I_{\wt{\Gam}}$ of the cup-product maps $c_\bbQ^\Gam, c_\bbQ^{\wt{\Gam}}$ we have:
\begin{equation}\label{eq:cupproduct}
\begin{tikzcd}
0 \arrow{r} & K_\Gam \arrow{r} \arrow{d}{p^*} & \bigwedge\nolimits^2 H^1(\Gam,\bbQ) \arrow{r}{c_\bbQ^\Gam} \arrow{d}{p^*} & I_\Gam \arrow{r}{\subset} \arrow{d}{p^*}  &H^2(\Gam,\bbQ)\arrow{d}{p^*}\\
0 \arrow{r} & K_{\wt{\Gam}} \arrow{r} & \bigwedge\nolimits^2 H^1(\wt{\Gam},\bbQ) \arrow{r}{c_\bbQ^{\wt{\Gam}}} & I_{\wt{\Gam}}\arrow{r}{\subset} & H^2(\wt{\Gam},\bbQ)
\end{tikzcd}
\end{equation}
Since the second vertical arrow is an isomorphism, it follows that the third vertical arrow is surjective: $I_{\wt{\Gam}} = p^*(I_\Gam)$. Since the kernel of the last vertical arrow is $\bbQ\, \om_\bbQ$, it follows that 
\begin{equation}\label{eq:jump}
p^*: I_\Gam/(I_\Gam\cap\bbQ\, \om_\bbQ) \lra I_{\wt{\Gam}}
\end{equation}
is an isomorphism.

Now we can prove the converse statement. Suppose our extension $E$ satisfies condition $N_2$. Let $N(\Gam), N(\wt{\Gam})$ denote the maximal $2$-step nilpotent quotients of $\Gam, \wt{\Gam}$ respectively. Then $\bbZ \le \wt{\Gam}$ must inject into $N(\wt{\Gam})$. These two nilpotent quotients are related as follows, where $Z( - )$ denotes the center of the respective groups:
\[
\begin{tikzcd}
0 \arrow{r} & Z(\wt{\Gam})\arrow{d}{p_*}\arrow{r} & N(\wt{\Gam})\arrow{d}{p_*}\arrow{r} & \wt{\Gam}^{ab}\arrow{d}{p_*}\arrow{r} & 0\\
0\arrow{r} & Z(\Gam)\arrow{r} & N(\Gam)\arrow{r} & \Gam^{ab} \arrow{r} &0
 \end{tikzcd}
\]
Since the central $\bbZ$ in $\wt{\Gam}$ maps trivially to $\Gam$, it has trivial image in $N(\Gam)$. Moreover, since the last vertical arrow is an isomorphism modulo torsion, a subgroup of finite index in $\bbZ$ maps trivially to $\wt{\Gam}^{ab}$, hence this infinite cyclic subgroup of $\bbZ$ injects into $Z(\wt{\Gam})$ but maps trivially to $Z(\Gam)$. Therefore the rank of $Z(\wt{\Gam})$ is strictly larger than the rank of $Z(\Gam)$.

Algebraic topology tells us that there are canonical isomorphisms
\begin{align*}
\Hom(Z(\Gam),\bbQ) &\cong K_\Gam & \Hom(Z(\wt{\Gam}),\bbQ) &\cong K_{\wt{\Gam}}
\end{align*}
where $K_\Gam,K_{\wt{\Gam}}$ are the kernels of the cup product maps in Equation \eqref{eq:cupproduct} (e.g., see \cite{Sullivan, BeauvilleDerived}). Letting lower case letters denote the dimensions of the corresponding $\bbQ$-vector spaces in Equation \eqref{eq:cupproduct}, we see that $k_\Gam + i_\Gam = k_{\wt{\Gam}} + i_{\wt{\Gam}}$, hence
\[
i_\Gam - i_{\wt{\Gam}} = k_{\wt{\Gam}} - k_\Gam.
\]
Since Equation \eqref{eq:jump} tells us that 
\[
i_\Gam - i_{\wt{\Gam}} = \begin{cases}
    1  & \text{if $\om\in I_\Gam $}, \\
      0 & \text{otherwise}.
\end{cases}
\]
we get, in particular, that if the rank of $Z(\wt{\Gam})$ is strictly larger than the rank $Z(\Gam)$, then $\om_\bbQ$ is in the image of the cup product, as desired.
\end{proof}

\section{Arithmetic lattices in $\PU(n,1)$}\label{sec:Arith}

\subsection{Congruence arithmetic lattices of simple type}\label{ssec:Simple}

The arithmetic lattices under consideration in this paper are those that are cocompact of \emph{simple type}. These are constructed as follows, where we generally follow the notation established in \cite[\S 6]{Acta}. Let $E / F$ be a totally imaginary quadratic extension of a totally real field with $[F : \bbQ] = d \ge 2$, let $\tau_1, \dots, \tau_d : E \to \bbC$ be representatives for the complex conjugate pairs of embeddings of $E$, and let $x \mapsto \conj{x}$ be the Galois involution of $E / F$, which extends to complex conjugation under any complex embedding of $E$.

Fix a nondegenerate hermitian vector space $V$ over $E$ of dimension $n+1$, and let $V_{\tau_j} \cong \bbC^{n+1}$ be the completion of $V$ with respect to the complex embedding $\tau_j$ of $E$. We assume that $V_{\tau_1}$ has signature $(n,1)$ and that $V_{\tau_j}$ is definite for $j \ge 2$. The assumption that $d \ge 2$ implies that there is at least one completion where $V$ is a definite hermitian space. It follows that $V$ is anisotropic.

The unitary group $\U(V)$ is a reductive algebraic group over $F$, and let $G_1 = \mathrm{Res}_{F / \bbQ} \U(V)$ be the restriction of scalars from $F$ to $\bbQ$. The above assumptions on the signature of $V$ at the various complex embeddings of $E$ imply that
\[
G_1(\bbR) \cong \U(n,1) \times \prod_{j \ge 2} \U(n+1).
\]
We can identify the ball $\bbB^n$ with $G_1(\bbR) / K_\infty$, where
\[
K_\infty \cong \left(\U(n) \times \U(1)\right) \times \prod_{j \ge 2} \U(n+1)
\]
is a chosen maximal compact subgroup of $G_1(\bbR)$.

Similarly, $G$ will denote the restriction of scalars from $F$ to $\bbQ$ of the group of unitary similitudes of $V$. If $\bbA_\bbQ$ denotes the adeles of $\bbQ$ and $\bbA_\bbQ^f$ the finite adeles, for each prime $p$ choose an open compact subgroup $K_p < G(\bbQ_p)$, and set $K = \prod K_p < G(\bbA_\bbQ^f)$. If the $K_p$ are chosen so that $K$ is open in $G(\bbA_\bbQ^f)$, we then obtain a Shimura variety
\[
S(K) = G(\bbQ) \bs \left(\bbB^n \times G(\bbA_\bbQ^f) \right) / K.
\]
As described in \cite[\S 6.4]{Acta}, $S(K)$ is a finite disjoint union of connected components $S(\Gam_j) = \Gam_j \bs \bbB^n$, where
\[
\Gam_j < G_{\textrm{ad}}(\bbR) \cong \PU(n,1) \times \prod_{j \ge 2} \PU(n+1)
\]
is a cocompact lattice. We will call lattices $\Gam_j$ of this kind \emph{congruence arithmetic lattices of simple type}. If we choose $K$ sufficiently small (e.g., neat) then each $\Gam_j$ is torsion-free and the ball quotient $S(\Gam_j)$ is both a manifold and a smooth projective variety. From this point forward, all $K$ are assumed to be sufficiently small in this sense.

\begin{rem}
Strictly speaking, we have only defined the \emph{cocompact} arithmetic lattices of simple type. One also obtains nonuniform lattices by taking $F = \bbQ$ in the above construction. For $n=2$, these are the classical Picard modular groups.
\end{rem}

\subsection{Cohomology of Shimura varieties}\label{ssec:Cohom}

In the notation from \S \ref{ssec:Simple}, if $K^\prime \le K < G(\bbA_\bbQ^f)$, then there is a finite map $S(K^\prime) \to S(K)$ with pullback
\[
H^*(S(K), R) \lra H^*(S(K^\prime), R)
\]
for any coefficient ring $R$. This allows us to define
\[
H^*(\mathrm{Sh}(G), R) = \varinjlim_K H^*(S(K), R).
\]
See \cite[\S 6.6]{Acta} for further details. For each $K$, we also have the Chern form $c_1(S(K)) \in H^2(S(K), \bbZ)$, which of course behaves well with respect to pull-back. In particular, we obtain a limiting class $\wt{c}_1 \in H^2(\mathrm{Sh}(G), \bbZ)$.

\subsection{Totally geodesic hypersurfaces}\label{ssec:Hyper}

We start with the following, which justifies restricting to lattices of simple type when considering arithmetic ball quotients containing a totally geodesic hypersurface. The case $n=2$ is found in \cite[\S 1]{MollerToledo}.

\begin{prop}\label{prop:TGSimple}
For $n \ge 2$, let $\Gam < \PU(n,1)$ be an arithmetic lattice so that $\Gam \bs \bbB^n$ contains a totally geodesic codimension one subvariety. Then $\Gam$ is of simple type.
\end{prop}

\begin{pf}[Sketch]
Suppose $\Gam < \PU(n,1)$ is arithmetic. Then there exists a totally imaginary quadratic field $E$, a central simple division algebra $\calD$ over $E$ of degree $d$ with involution $\sig$ of the second kind, and a $\sig$-hermitian form $h$ on $\calD^r$ so that the algebraic group $\calG$ associated with $\Gam$ is the projective unitary group of $h$. Note that $n+1 = d r$. The case where $\Gam$ has simple type is precisely the case $d = 1$ (i.e., $\calD = E$) and $r = n+1$.

If $\Gam \bs \bbB^n$ contains a totally geodesic subvariety of the form $\Lam \bs \bbB^{n-1}$, it is well known that $\Lam$ is also arithmetic. Then there is a totally imaginary subfield $E_0$ of $E$, a subalgebra $\calD_0 \subseteq \calD$ over $E_0$ of degree $d_0$ dividing $d$, and an embedding of $\calD_0^{r_0}$ in $\calD^r$ so that $d_0 r_0 = n$ and the restriction of $h$ to the image of $\calD_0^{r_0}$ is the algebraic group associated with $\Lam$. The fact that $\calG$ is noncompact at exactly one place of $E$ forces $E_0 = E$. Now, $d_0$ divides $n$ and $n+1$, which forces $d_0 = 1$. Then $r_0 = n$, but $r_0 \le r \le n+1$, so $r$ is $n$ or $n+1$. However, $n \ge 2$ and $r$ divides $n+1$, so $r = n+1$, hence $d=1$, which proves the proposition.
\end{pf}

One often classifies the immersed totally geodesic complex codimension one subvarieties of arithmetic ball quotients $\Gam \bs \bbB^n$ by $\Gam$-orbits of lines $\ell \subset V$ so that the restriction of the hermitian form to $\ell$ is positive under $\ell \hookrightarrow V_{\tau_1}$ or, dually, by codimension one subspaces $W \subset V$ so that $W_{\tau_1}$ has signature $(n-1,1)$. We instead follow the notation of Kudla--Millson \cite[\S 2\,\&\,9]{KM} and Bergeron--Millson--Moeglin \cite[\S 8]{Acta}, since it will be used for the proof of our main results.

For each $d \times d$ hermitian matrix $\beta \in \M_d(E)$ that is positive definite at $\tau_1$, Kudla and Millson define a \emph{special cycle} $C_\beta$. This is a locally finite cycle consisting of totally geodesic subvarieties of (complex) codimension $d$ \cite[p.~132]{KM}. Moreover, all immersed totally geodesic subvarieties are contained in the support of some $C_\beta$. See the above references for further technical details, since they will not concern us in this paper. As pointed out to us by Nicolas Bergeron, the following result has the exact same proof as \cite[Cor.~8.4]{Invent}, except that paper covers the orthogonal case. We give the argument for completeness.

\begin{thm}\label{thm:ChernInSpan}
Let $X = \Gam \bs \bbB^n$ be a smooth compact ball quotient, $n \ge 2$, with $\Gam$ a congruence arithmetic lattice of simple type, and let $c_1$ be the Chern form on $X$, considered as an element of $H^2(X, \bbC)$. Then $c_1$ is contained in the span of the collection of all Poincar\'e duals to codimension one totally geodesic subvarieties of $X$.
\end{thm}

\begin{pf}
Suppose not. Then there exists $\eta \in H^2(X, \bbC)^* = H^{2(n-1)}(X, \bbC)$ such that:
\begin{align*}
\int_X \eta \wedge c_1 &\neq 0 \\
\int_C \eta &= 0 & \textrm{for all totally geodesic}~C~\textrm{of codimension}~1
\end{align*}
In other words, the periods of $\eta$ over all codimension $1$ geodesic subvarieties are zero.

Following, \cite[p.~126]{KM}, let $\calL(t)$ be the space of hermitian $1 \times 1$ matrices of rank $t \le 1$ over the ring of integers $\calO$ of the totally imaginary field $E$ from the definition of the arithmetic lattice $\Gam$. Then $\calL(1)$ is the ring of integers $\calO_F$ of the maximal totally real subfield $F$ of $E$ and $\calL(0) = \{0\}$. For each $\beta \in \calL(t)$, we have the totally geodesic cycle $C_\beta$ of codimension $t$ in $X$, where $C_0 = X$ and for $\beta$ in the nonzero elements $\calO_F^\times$ of $\calO_F$ we have $C_\beta$ of codimension $1$.

For $\beta \in \calL(1)$ and $\zeta \in \bbC^{[F : \bbQ]}$, following \cite[p.~125]{KM}, define
\[
e_*(\beta \zeta) = \prod_{j = 1}^{[F : \bbQ]} \exp(\pi i \tr_{\bbC / \bbR}(\tau_j(\beta)) \zeta_{\tau_j}),
\]
where the product is over the places $\{\tau_j\}$ of $E$ and the coordinates of $\zeta$ are labelled accordingly. Note that $e_*(0 \zeta) = 1$ for all $\zeta$. As on \cite[p.~127]{KM} we then have that
\begin{align*}
P(\zeta, \eta) &= \sum_{t = 0}^1 \sum_{\beta \in \calL(t)} \left( \int_{C_\beta} \eta \wedge c_1^{1-t} \right) e_*(\beta \zeta) \\
&= e_*(0 \zeta) \int_X \eta \wedge c_1 + \sum_{\beta \in \calO_F^\times} e_*(\beta \zeta) \int_{C_\beta} \eta \\
&= \int_X \eta \wedge c_1
\end{align*}
is a constant function of $\zeta$ with value $\int_X \eta \wedge c_1 \neq 0$, since $\int_{C_\beta} \eta = 0$ for all $\beta$ by hypothesis.

However, by \cite[Thm.~3]{KM}, the function $P(\zeta, \eta)$ of $\zeta$ is a holomorphic modular form of weight $n+1$ for a suitable congruence subgroup of the group of $\calO$-points of a certain form of $\U(1,1)$ over $F$. Since a nonzero constant function is not a holomorphic modular form of higher weight, this contradiction proves the theorem.
\end{pf}

\subsection{The cup product and special cycles}\label{ssec:SpecialCup}

The purpose of this section is to prove the following theorem, whose proof was communicated to us by Nicolas Bergeron.

\begin{thm}\label{thm:Cup}
Let $X = \Gam \bs \bbB^n$ be a compact congruence arithmetic ball quotient of simple type and $\mathrm{SC}^1(X) \subseteq H^2(X, \bbC)$ be the span of the Poincar\'e duals to the totally geodesic codimension one subvarieties of $X$. Then:
\begin{enum}

\item For any class $\sig \in \mathrm{SC}^1(X)$, there is a congruence cover $p : X^\prime \to X$ so that $p^*(\sig) \in \mathrm{SC}^1(X^\prime)$ is in the image of the cup product map
\[
c : \bigwedge\nolimits^2 H^1(X^\prime, \bbC) \lra H^2(X^\prime, \bbC).
\]

\item If $n \ge 3$, then for every $\phi \in H^{1,1}(X, \bbC)$ we can find a congruence cover $p : X^\prime \to X$ so that $p^*(\phi)$ is contained in the image of $c$.

\item If $n \ge 4$, then for all $\phi \in H^2(X, \bbC)$ we can find a congruence cover $p : X^\prime \to X$ so that $p^*(\phi)$ is in the image of $c$.

\end{enum}

\end{thm}

\begin{pf}
Let $G$ be the $\bbQ$-algebraic group associated with $X$ and $\Gam$ as in \S\ref{ssec:Simple}, and suppose that $X$ is a connected component of the Shimura variety $S(K)$. The theorem follows fairly directly from deep structural results on cohomology arising from the theta correspondence; see \cite[\S 7]{Acta} for a precise definition.

Taking $q=1$ and the appropriate values of $(a,b)$ in \cite[Cor.~7.3]{Acta}, we see that $H^1(X, \bbC)$ is generated by the classes of theta lifts for all $n \ge 2$, $H^{1,1}(X, \bbC)$ is generated by theta lifts when $n \ge 3$, and all of $H^2(X, \bbC)$ is generated by the theta correspondence for $n \ge 4$. For $n \ge 2$, Kudla and Millson proved that the Poincar\'e duals to codimension $1$ geodesic subvarieties are contained in the subspace generated by theta lifts and Poincar\'e duals to special cycles span the subspace generated by theta lifts \cite{KMI, KMII, KM} (also see \cite[p.~9]{Acta}).

Now, suppose that $\sig \in H^2(X, \bbC)$ is contained in the subspace generated by theta lifts. The above discussion implies that this is the case for every class satisfying one of the three conditions from the statement of the theorem. Let
\[
\wt{\sig} \in H^*(\mathrm{Sh}(G), \bbC) = \varinjlim_K H(S(K), \bbC)
\]
be the image of $\sig$ in $H^2(\mathrm{Sh}(G), \bbC)$ as in \S\ref{ssec:Cohom}. Using \cite[Prop.~5.4 \& 5.19]{Acta}, arguing almost verbatim as in the proof of \cite[Thm.~9.3]{Acta} (the only change being the degrees of the forms under consideration), we have that $\wt{\sig}$ is in the image of the cup product map
\[
\bigwedge\nolimits^2 H^1(\mathrm{Sh(G)}, \bbC) \lra H^2(\mathrm{Sh(G)}, \bbC).
\]
This implies that there is a congruence cover $p : X^\prime \to X$ so that
\[
p^*(\sig) \in \mathrm{Im}\left(\bigwedge\nolimits^2 H^1(X^\prime, \bbC) \lra H^2(X^\prime, \bbC) \right),
\]
which completes the proof of the theorem.
\end{pf}

\begin{rem}\label{rem:OverQ}
By the universal coefficients theorem, Theorem~\ref{thm:Cup} is also true with coefficients in any subfield of $\bbC$. In particular, it holds over $\bbQ$.
\end{rem}

Theorem~\ref{thm:ChernInSpan} and Theorem~\ref{thm:Cup} combine with Remark~\ref{rem:OverQ} to give the following.

\begin{cor}\label{cor:ChernCover}
Let $X$ be a smooth compact congruence arithmetic ball quotient of simple type with dimension $n \ge 2$. Suppose that $\phi \in H^2(X, \bbZ)$ is one of the following classes:
\begin{enum}

\item $c_1(X)$;

\item $c_1(D)$ for $D \subset X$ a reduced effective divisor with support a union of totally geodesic subvarieties, or any linear combination of such classes;

\item any class in $H^{1,1}(X, \bbQ) \cap H^2(X, \bbZ)$ if $n \ge 3$;

\item an arbitrary element of $H^2(X, \bbZ)$ if $n \ge 4$.

\end{enum}
Then there is a congruence cover $p : X^\prime \to X$ so that the image of $p^*(\phi)$ in $H^2(X^\prime, \bbQ)$ is contained in the image of the cup product map
\[
c_\bbQ : \bigwedge\nolimits^2 H^1(X^\prime, \bbQ) \lra H^2(X^\prime, \bbQ).
\]
\end{cor}

\begin{rem}
Passage to a finite cover in Corollary~\ref{cor:ChernCover} is necessary. For example, let $X$ be the Cartwright--Steger surface. Then $X$ is the quotient of $\bbB^2$ by a congruence arithmetic lattice of simple type (e.g., see \cite[p.~90]{StoverHurwitz}). Then $H^1(X, \bbZ) \cong \bbZ^2$, so the image of the cup product in $H^2(X, \bbZ)$ is generated by the primitive element $\al \wedge \beta$, where $\al$ and $\beta$ span $H^1(X, \bbZ)$. Since $c_1(X)^2 \neq 0$ but $(\al \wedge \beta)^2 = 0$, we see that $c_1(X)$ cannot be in the image of the cup product map. As shown by Dzambic and Roulleau \cite{DzambicRoulleau}, the Stover surface, a $21$-fold \'etale cover of $X$ with fundamental group a congruence arithmetic lattice, satisfies the conclusion of Corollary~\ref{cor:ChernCover}.
\end{rem}

\section{Proofs of Theorems~\ref{thm:MainRF} and \ref{thm:MainRF2}}\label{sec:Proofs}

We begin by proving Theorem~\ref{thm:MainRF2}, then use it to prove Theorem~\ref{thm:MainRF} as sketched in the introduction.

\begin{pf}[Proof of Theorem~\ref{thm:MainRF2}]
Suppose that $\Gam < \PU(n,1)$, $\phi \in H^2(\Gam \bs \bbB^n, \bbZ)$, and $\wt{\Gam}$ satisfy the hypotheses of the theorem. By \S\ref{subsub:extensions} (in particular Remark~\ref{rem:ext}) and Lemma~\ref{lem:compare}, if $\Gam^\prime \le \Gam$ is a finite index subgroup, then the preimage $\wt{\Gam}^\prime$ of $\Gam^\prime$ in $\wt{\Gam}$ is the central extension of $\Gam^\prime$ with characteristic class $p^*(\phi)$, where $p : \Gam^\prime \bs \bbB^n \to \Gam \bs \bbB^n$ is the associated cover. Thus $\wt{\Gam}^\prime$ is a finite index subgroup of $\wt{\Gam}$, and by Lemma~\ref{lem:behavior} it suffices to prove that $\wt{\Gam}^\prime$ and the associated reductions $\wt{\Gam}_d^\prime$ of $\wt{\Gam}^\prime$ modulo $d$ are residually finite.

By Corollary~\ref{cor:ChernCover}, we can choose $\Gam^\prime$ to be a congruence subgroup of $\Gam$ such that
\[
p^*(\phi) \in \mathrm{Im}\left(\bigwedge\nolimits^2 H^1(X^\prime , \bbQ) \lra H^2(X^\prime, \bbQ) \right),
\]
where $X^\prime = \Gam^\prime \bs \bbB^n$. By Theorem~\ref{thm:CentralRF1}, the extension with group $\wt{\Gam}^\prime$ satisfies condition $N_2$. Since $\Gam^\prime < \PU(n,1)$ is residually finite, $\wt{\Gam}^\prime$ is residually finite by Lemma~\ref{lem:condn} and $\wt{\Gam}_d$ is residually finite by Lemma~\ref{lem:behavior}. Thus $\wt{\Gam}$ and $\wt{\Gam}_d$ are also residually finite, and this completes the proof of the theorem.
\end{pf}

We now deduce Theorem~\ref{thm:MainRF} from Theorem~\ref{thm:MainRF2}.

\begin{pf}[Proof of Theorem~\ref{thm:MainRF}]
Let $\Gam < \PU(n,1)$ be a cocompact arithmetic lattice of simple type and $\wt{\Gam}$ its preimage in $\wt{\PU}(n,1)$. If $n=1$, then it is classical that $\wt{\Gam}$ is residually finite; see the introduction to \cite{StoverToledo} for discussion. If $n \ge 2$, first note that if $\Lam$ is commensurable with $\Gam$, then the preimage $\wt{\Lam}$ of $\Lam$ in $\wt{\PU}(n,1)$ is commensurable with $\wt{\Gam}$ (cf.\ Remark~\ref{rem:ext}). Therefore $\wt{\Gam}$ is residually finite if and only if $\wt{\Lam}$ is residually finite, hence we can assume without loss of generality that $\Gam$ is a torsion-free congruence arithmetic lattice.

Now, $\wt{\Gam}$ is the central extension of $\Gam$ with characteristic class $c_1(\Gam \bs \bbB^n)$. Moreover, the reduction $\wt{\Gam}_d$ of $\wt{\Gam}$ modulo $d$ is isomorphic to the preimage of $\Gam$ in the $d$-fold connected cover of $\PU(n,1)$. By Theorem~\ref{thm:ChernInSpan}, we have that $c_1(\Gam \bs \bbB^n)$ is contained in the span of the Poincar\'e duals to the totally geodesic cycles on $\Gam \bs \bbB^n$, so Theorem~\ref{thm:MainRF2} applies to give residual finiteness of $\wt{\Gam}$ and $\wt{\Gam}_d$. This completes the proof.
\end{pf}

\section{Branched covers of ball quotients}\label{sec:Branched}

We begin with the key application of our main results, which provides existence of the branched covers used to prove Theorem~\ref{thm:NegativeExist}.

\begin{prop}\label{prop:GoodCover}
Fix $d \ge 2$. Let $X = \Gam \bs \bbB^n$ be a smooth compact complex hyperbolic $n$-manifold, $n \ge 2$, with $\Gam$ a congruence arithmetic lattice of simple type. Then we can pass to a congruence cover $X^\prime = \Gam^\prime \bs \bbB^n$ so that $X^\prime$ contains a totally geodesic divisor $D^\prime$ such that:
\begin{enum}
\item $(X^\prime, D^\prime)$ is a good pair in the sense of Definition~\ref{def:goodpair}, and
\item there exists a smooth projective variety $Y$ and a cyclic branched cover $f : Y \to X^\prime$ with branch divisor $D^\prime$ and branching of degree $d$ along each irreducible component of $D^\prime$.
\end{enum}
\end{prop}

\begin{pf}
As described in \S\ref{ssec:Hyper}, it is well known that $X$ contains an \emph{immersed} totally geodesic codimension one ball quotient $D$. It is a standard fact that we can replace $X$ with a finite congruence cover and assume that $D$ is in fact embedded, hence $(X, D)$ is a good pair in the sense of Definition~\ref{def:goodpair}. Corollary~\ref{cor:ChernCover} and Theorem~\ref{thm:CentralRF1} imply that condition $N_2$ is satisfied by the central extension of $\Gam$ by $\bbZ$ associated with $\pi_1(\calO(D)^\times)$ as in \S\ref{subsec:basic}. Therefore, by Corollary~\ref{cor:existbranch} there is a further finite cover $p : X^\prime \to X$ so that there exists a cyclic $d$-fold branched cover $Y\to X^\prime$, branched over the totally geodesic divisor $D^\prime = p^*(D)$, with $Y$ smooth.
\end{pf}

We will apply the following theorem of Zheng \cite{Zheng}, where we restate Zheng's notion of a ``good cover'' with smooth totally geodesic branch divisor in our language. His condition \emph{negative definite complex curvature operator} is the same as Siu's very strong negative curvature, which implies strong negative curvature, which in turn implies negative sectional curvature. See \cite[\S 2]{Siu}.

\begin{thm}[Thm.~1 \cite{Zheng}]\label{thm:Zheng}
Let $f : Y \to X$ be a $d$-fold cyclic branched cover of a smooth compact ball quotient $X$ with branch divisor $D$ a disjoint union of smooth totally geodesic codimension one subvarieties and branching of order $d$ along each component. Then $Y$ admits a K\"ahler metric with negative definite complex curvature operator. In fact, the metric on $Y$ is strongly negative in the sense of Siu, hence $Y$ is strongly rigid.
\end{thm}

We now prove Theorem~\ref{thm:NegativeExist}.

\begin{pf}[Proof of Theorem~\ref{thm:NegativeExist}]
Let $Y$ be one of the smooth projective varieties provided by Proposition~\ref{prop:GoodCover}. Then $Y$ admits a metric of strongly negative curvature by Theorem~\ref{thm:Zheng}, so to complete the proof we must show that $Y$ is not homotopy equivalent to a locally symmetric manifold. Since $Y$ is a closed manifold admitting a complete metric of negative curvature, it is well known that $\pi_1(Y)$ cannot contain a free abelian subgroup of rank $r \ge 2$. A cocompact lattice in a semisimple Lie group of real rank $r$ contains a subgroup isomorphic to $\bbZ^r$ by a theorem of Wolf \cite[Thm.~4.2]{Wolf}, hence we only need to rule out the possibility that $\pi_1(Y)$ is isomorphic to a lattice in an adjoint simple Lie group of real rank one.

Let $Z$ be a rank one locally symmetric space homotopy equivalent to $Y$. We can exclude $Z$ being real, quaternionic, or Cayley hyperbolic by a result of Carlson and Toledo \cite[Cor.~3.3]{CarlsonToledo}: since $Y$ is a compact K\"ahler manifold, it cannot be homotopy equivalent to a locally symmetric space $Z$ unless $Z$ is locally Hermitian symmetric. This and real rank one implies that $Z$ is a ball quotient. Considering cohomological dimension we see that $\dim(Y) = \dim(Z)$.

If $f : Y \to Z = \Lam \bs \bbB^n$ is a homotopy equivalence, and $g:Z\to Y$ a homotopy inverse to $f$, since both $Y$ and $Z$ have strongly negative curvature, Siu rigidity \cite[Thm.~1]{Siu} implies that we can assume that both $f$ and $g$ are holomorphic maps. Thus $f \circ g$ and $g \circ f$ are holomorphic self maps of $Z$ and $Y$, respectively, that are homotopic to the identity, hence they are the identity. It follows that $f$ is a biholomorphism and $g=f^{-1}$. See also  \cite[Thm.~2]{Siu}

By definition of a branched cover, $\bbZ / d$ acts on $Y$ by holomorphic automorphisms with quotient the ball quotient $X^\prime$ from the statement of Proposition~\ref{prop:GoodCover}. Let $\rho \in \Aut(Y)$ be a generator for this action and $D^\prime$ be the branch locus, which we recall is a disjoint union of smooth totally geodesic subvarieties of $X^\prime$. If $\calD$ denotes the preimage of $D^\prime$ on $Y$, then $\calD$ is a divisor with support a disjoint union of smooth subvarieties of $Y$, each isomorphic to a smooth totally geodesic subvariety $D_i \subset X^\prime$. The fixed point set of $\rho$ is $\calD$, and the cyclic group generated by $\rho$ acts freely on $Y\ssm \calD$.

Consider the holomorphic automorphism
\[
\sig = \left(f \circ \rho \circ f^{-1}\right) : Z \to Z
\]
of period $d$ fixing the divisor $f(\calD)$. Since automorphisms of $Z$ are isometries for its Bergman metric and fixed sets of isometries are totally geodesic, it follows that $f(\calD)$ is a totally geodesic divisor on $Z$ and the support of $f(\calD)$ consists of a disjoint union of smooth subvarieties of $Z$ isomorphic to the ball quotient $D_i$. In particular, each $D_i$ is isometric to a smooth totally geodesic subvariety of both $X^\prime$ and $Z$. Moreover, since $f$ is a biholomorphism, the normal bundle to $D_i$ in $Z$ is isomorphic to the normal bundle of $D_i$ in $Y$. We use this to derive a contradiction.

If $X$ is an $n$-dimensional ball quotient and $D_i$ a smooth totally geodesic subvariety of codimension one, then the normal bundle to $D_i$ in $X$ has first Chern class $\nu$ satisfying $n \nu = c_1(D_i)$ in $H^2(D_i, \bbQ)$. One way to see this is to use the principle of proportionality in dual symmetric spaces: any relation between characteristic classes with $\bbQ$ coefficients that holds in one element of a dual pair must also hold in the other. For $\bbP^{n-1} \subset \bbP^n$ this equality is easily checked. Indeed, the normal bundle is the bundle $\nu = \calO (1)$ with ${c_1(\nu) = x}$, where $x$ is a positive generator of $H^2(\bbP^{n-1},\bbZ)$, while ${c_1(\bbP^{n-1}) = nx}$, thus the equivalent equality must also hold for a geodesic divisor in a ball quotient.

In particular, the first Chern class of the normal bundle to $D_i$ is completely independent of the ambient ball quotient. Applying this to $X^\prime$ and considering the branched cover $Y \to X^\prime$, the normal bundle to $D_i$ in $Y$ has first Chern class
\[
d \nu \in H^2(D_i, \bbQ),
\]
hence the same holds for the normal bundle to $D_i$ in $Z$. This implies that $c_1(D_i) = d c_1(D_i)$, but $d \neq 1$ and $c_1(D_i)$ is nontrivial, hence we arrive at a contradiction. This proves that $Y$ cannot be homotopy equivalent to a ball quotient, which completes the proof of the theorem.
\end{pf}

\bibliography{ArithmeticRF}

\end{document}